\documentclass{amsart}

\usepackage{amssymb}
\usepackage[all]{xy}

\newcommand{\fr}[3]{\ensuremath{{#1}_{#2}\langle{#3}\rangle}}
\newcommand{\ff}[3]{\ensuremath{{#1}_{#2}(\langle{#3}\rangle)}}

\DeclareMathOperator{\Aut}{Aut}

\DeclareMathOperator{\im}{im}
\DeclareMathOperator{\op}{op}
\DeclareMathOperator{\supp}{supp}
\DeclareMathOperator{\charac}{char}

\theoremstyle{plain}
\newtheorem{theorem}{Theorem}[section]
\newtheorem{proposition}[theorem]{Proposition}
\newtheorem{lemma}[theorem]{Lemma}
\newtheorem{corollary}[theorem]{Corollary}

\begin{document}

\title{Free fields in skew fields}

\author[V. O. Ferreira]{Vitor O. Ferreira}
\address{Department of Mathematics - IME, University of S\~ao Paulo,
Caixa Postal 66281, S\~ao Paulo, SP, 05314-970, Brazil}
\email{vofer@ime.usp.br}
\thanks{The first author is the corresponding author and was partially supported by CNPq, Brazil (Grant
302211/2004-7).}

\author[E. Z. Fornaroli]{\'Erica Z. Fornaroli}
\address{Department of Mathematics, State University of Maring\'a,
Avenida Colombo, 5790, Maring\'a, PR, 87020-900, Brazil}%
\email{ezancanella@uem.br}
\thanks{The second author was supported by CNPq, Brazil (Grant 141505/2005-2)}



\date{27 May 2008}

\begin{abstract}
  Building on the work of K.~Chiba (\textit{J.~Algebra}~\textbf{263} (2003), 75--87), we
  present sufficient conditions for the completion of a division ring
  with respect to the metric defined by a discrete valuation function to
  contain a free field, \emph{i.e.} the universal field of fractions of a free associative
  algebra. Several applications to division rings generated by
  torsion-free nilpotent groups, skew Laurent series and related division
  rings are discussed.
\end{abstract}

\maketitle


\section*{Introduction}

In this paper we present a slight improvement on a theorem
of K.~Chiba~\cite{kC03} on the existence of free fields in division rings
which are complete with respect to a valuation function. We then
move on to discuss large classes of examples of division rings
that do and do not contain free fields.

In what follows the terms ``division ring'' and ``skew field'' will
be used interchangeably. Occasionally, we shall omit the adjective
``skew''. A skew field which is commutative will always be called
a ``commutative field''.

Let us recall from \cite[Chapter 7]{pC85} that given a ring $R$, by
an \emph{$R$-ring} one understands a ring $L$ with a homomorphism
$R\rightarrow L$. For fixed $R$, the $R$-rings form a category in
which the maps are the ring-homomorphisms $L\rightarrow L'$ such
that the triangle below is commutative.
 $$
    \xymatrix{
      & R\ar[dl] \ar[dr] & \\
     L\ar[rr] & & L'}
 $$
An $R$-ring which is a skew field is called an \emph{$R$-field}. An
$R$-field $K$ is called an \emph{epic} $R$-field if $K$ is generated
as field by the image of $R$. If $K$ is an epic $R$-field for which
the map $R\rightarrow K$ is injective, $K$ is called a \emph{field
of fractions} of $R$.

By a \emph{local homomorphism} between $R$-fields $K$, $L$ one
understands an $R$-ring homomorphism $\alpha\colon K_0\rightarrow L$
from a subring $K_0$ of $K$ containing the image of $R$ into $L$
such that $K_0\setminus\ker\alpha \subseteq U(K_0)$, where $U(K_0)$
denotes group of units of $K_0$. A \emph{specialization} from $K$ to
$L$ is an equivalence class of local homomorphisms from $K$ into
$L$, where two local homomorphisms $\alpha$ and $\beta$ with domains
$K_0$ and $K_1$, respectively, are equivalent if there exists a
subring $K_2$ containing the image of $R$ such that $K_2 \subseteq
K_0\cap K_1$, $\alpha(x)=\beta(x)$, for all $x\in K_2$, and
$\alpha|_{K_2}=\beta|_{K_2}\colon K_2\rightarrow L$ is a local
homomorphism.

A \emph{universal} $R$-field is an epic $R$-field $U$ such that for
any epic $R$-field $K$ there exists a unique specialization from $U$
to $K$. Such a universal $R$-field, if it exists, is unique up to
isomorphism. If $R$ has a universal $R$-field $U$ then $R$ has a
field of fractions if and only if $U$ is its field of fractions. In
that case $U$ is called the \emph{universal field of fractions} of
$R$.

Given a skew field $D$, a subfield $K$ of $D$ and a set $X$, we
define the \emph{free $D_K$-ring} on $X$ to be the ring
$\fr{D}{K}{X}$ generated by $X$ over $D$ satisfying the relations
$ax=xa$ for all $a\in K$ and $x\in X$. It is well known that
$\fr{D}{K}{X}$ is a fir and, thus, has a universal field of
fractions $\ff{D}{K}{X}$, called the \emph{free $D_K$-field} on $X$.
(We refer to \cite{pC85} for the theory of firs and their universal
fields of fractions.) The free $D_D$-ring on $X$ will be called the
free $D$-ring on $X$ and will be denoted by $\fr{D}{}{X}$. Its
universal field of fractions, the free $D$-field, will be denoted by
$\ff{D}{}{X}$. (When $D$ is commutative, the free $D$-ring on $X$ is
just the free $D$-algebra on $X$.) When we refer to a free field, we
shall mean a $D_K$-free field on some nonempty set $X$ for some skew
field $D$ and subfield $K$.
%
%
%
%

We shall be looking at valuation functions on skew fields. We follow the
notation and definitions of \cite{kC03}. Let $R$ be a ring and
let $G$ be a (not necessarily abelian) ordered group with operation denoted additively.
By a \emph{valuation} on $R$ with values in $G$ one understands a map
$\nu\colon R\rightarrow G\cup\{\infty\}$ satisfying
\renewcommand{\theenumi}{\roman{enumi}}
\begin{enumerate}
  \item $\nu (xy)=\nu(x)+ \nu(y)$, for all $x, y \in R$,\label{val1}
  \item $\nu(x+y)\geq \min\{\nu(x), \ \nu(y)\}$, for all $x, y \in R$,
  \item $\nu(1)=0$ and $\nu(0)=\infty$,
\end{enumerate}
where it is understood that $a<\infty$ and $a+\infty=
\infty+a=\infty+\infty=\infty$, for all $a\in G$.
A valuation $\nu$ on $R$ will be called \emph{proper} if
the ideal $\nu^{-1}(\infty)$ is trivial; for example, on a field every
valuation is proper. All valuations considered in this work will be
proper.

Let $D$ be a skew field with a valuation $\nu\colon D\rightarrow
G\cup\{\infty\}$ and let $D^{\times}$ denote the set of all nonzero elements of $D$.
Then $\nu(D^{\times})$ is a subgroup of
$G$ called the \emph{value group} of $\nu$. If the ordered group
$\nu(D^{\times})$ is isomorphic to the additive group of integers,
then $\nu$ is called \emph{discrete}. When $\nu(D^{\times})$ is a ordered
subgroup of the additive group of real numbers, for instance, when $\nu$ is
discrete, we can define a metric $d$ on $D$ by choosing a real constant $c\in (0, 1)$
and letting
$$
d(x,y)=c^{\nu(x-y)},\quad\text{for all $x,y\in D$}.
$$
The topology so defined on $D$ is independent of the choice of the constant $c$ and the
completion $\widehat{D}$ of the metric space $D$ is again a skew field with a valuation
$\hat{\nu}$ such that $\widehat{D}$ and $\hat{\nu}$ are extensions of $D$ and $\nu$,
respectively.
We shall refer to $\widehat{D}$ as the completion of
$D$ with respect to the metric induced by $\nu$ or simply as the
completion of $D$ with respect to the valuation $\nu$.

An example is the following. Let $R$ be a ring with a central
regular element $t$ such that $\bigcap{t^nR}=0$ and $R/tR$ is a
domain. If we define, for each $x\in R$, $\nu(x)=\sup\{n : x\in
t^nR\}$ then $\nu$ is a valuation on $R$, called the \emph{$t$-adic
valuation}. In particular, if $R$ is a domain and if $R[z]$ is the
polynomial ring over $R$ on the indeterminate $z$, $R[z]$ has a
$z$-adic valuation. If, moreover, $R$ is a division ring, then
$R[z]$ is an Ore domain with field of fractions $R(z)$, the field of
rational functions on $z$. The $z$-adic valuation on $R[z]$ extends
to a discrete valuation $\nu$ on $R(z)$ and the completion of $R(z)$
with respect to $\nu$ is just the field of Laurent series $R((z))$.
We shall consider more general situations of this instance.

\medskip

Section~\ref{sec:gendr} of the present article is devoted to showing that Chiba's result on the
existence of free fields in completions of valued division rings which
are infinite dimensional over their centres holds for arbitrary
division rings with infinite centres. Subsequently, we propose
a method for guaranteeing the existence of free fields in
division rings which have a specialization into division rings
with free fields.

Section~\ref{sec:noff} elaborates on a commentary of
A.~Lichtman on division rings which do not contain free fields,
although their are known to contain free algebras.

The main results of the paper are in the last three sections.
Section~\ref{sec:mn} treats the case of division rings of
skew Malcev-Neumman series defined on a torsion-free nilpotent
group $G$ over a division ring $K$. They are shown to contain the
completion of the division ring of fractions of the group
ring of $G$ over $K$ with respect to a very natural valuation.
Free fields are shown to exist in both large division rings.

In Section~\ref{sec:ls} we look into division rings of skew Laurent
series over a division ring. Being complete with respect to natural
valuations (either ``$t$-adic'' or ``at infinity'', depending on
the presence of derivations), they are subjected to the criteria
of Section~\ref{sec:gendr}. Special instances considered are
appropriate completions of the Weyl field and of the division
ring generated by $2\times 2$ quantum matrices.

Finally, in Section~\ref{sec:adicval}, more general $t$-adic
valuations are considered. In particular, we obtain a
completion of the field of fractions of the universal
enveloping algebra of an arbitrary Lie algebra in
characteristic zero containing a free field.

\renewcommand{\theenumi}{\arabic{enumi}}

\section{Free fields in valued division rings}\label{sec:gendr}

\noindent





The proof of the following result can be extracted from the
proof of \cite[Theorem~1]{kC03}.

\begin{theorem}\label{th:chiba}
  Let $D$ be a skew field with infinite centre and let $K$ be a subfield of
  $D$ which is its own bicentralizer and whose centralizer $K'$ in $D$ is
  such that the left $K$-space $KcK'$ is infinite-dimensional,
  for all $c\in D^{\times}$. Suppose that $\nu$ is a discrete valuation
  on $D$ such that there exists a nonzero element $t$ of $K'$ with $\nu(t)>0$.
  Then for every countable set $\Sigma$ of full matrices over the
  free $D_K$-ring $\fr{D}{K}{X}$ there exists a
  $\Sigma$-inverting homomorphism from $\fr{D}{K}{X}$ into
  the completion $\widehat{D}$ of $D$ with respect to the valuation $\nu$.\hfill\qed
\end{theorem}

(In particular, it follows that when $D$ and $X$ are countable and
$\Sigma$ is the set of all full matrices over $\fr{D}{K}{X}$ the
completion $\widehat{D}$ contains the free field $\ff{D}{K}{X}$.
That is the statement of \cite[Theorem~1(1)]{kC03}.)

The following consequence of Theorem~\ref{th:chiba} should be compared
with \cite[Corollary~1(1)]{kC03}.

\begin{theorem}\label{th:freefield}
  Let $D$ be a skew field with infinite centre $C$ such that
  the dimension of $D$ over $C$ is infinite. If there exists a discrete
  valuation $\nu$ on $D$ then the completion $\widehat D$ of $D$
  with respect to the metric induced by $\nu$ contains a free
  field $\ff{C}{}{X}$ on a countable set $X$.
\end{theorem}

\begin{proof}
  Let $X$ be a countable set and let $F$ be a countable subfield of $C$.
  Let $\Sigma$ denote the set of all full matrices over $\fr{F}{}{X}$.
  Since both $F$ and $X$ are countable, $\Sigma$ is countable. Moreover,
  $\Sigma$ is a set of full matrices over $\fr{D}{C}{X}$ because the natural
  inclusion $\fr{F}{}{X} \hookrightarrow \fr{D}{C}{X}$
  is an honest map, by \cite[Theorem~6.4.6]{pC95}. It follows
  from Theorem~\ref{th:chiba} that there exists a $\Sigma$-inverting
  homomorphism $\fr{D}{C}{X} \rightarrow \widehat{D}$ and, therefore,
  the composed map
  $$
    \fr{F}{}{X}\hookrightarrow\fr{D}{C}{X}\longrightarrow\widehat{D}
  $$
  is a $\Sigma$-inverting $F$-ring homomorphism which extends to an $F$-ring homomorphism
  $\ff{F}{}{X}\rightarrow\widehat{D}$. Thus $X$ freely generates a free
  subfield of $\widehat{D}$ over $F$. Since $F$ is a central subfield of $D$,
  $X$ freely generates a free subfield of $\widehat{D}$ over the prime field of $D$ and,
  \textit{a fortiori}, over $C$, by \cite[Lemma 9]{kC03}.
\end{proof}

This provides a more direct proof the following version of \cite[Corollary~1(2)]{kC03}.

\begin{corollary}
  Let $D$ be a skew field with centre $C$ such that the
  dimension of $D$ over $C$ is infinite. Then the skew field of
  Laurent series $D(( z))$ in $z$ over $D$ contains a free field
  $\ff{C}{}{X}$ on a countable set $X$.
\end{corollary}

\begin{proof}
  Let $\omega$ denote the valuation on $D(z)$ which extends the $z$-adic
  valuation on $D[z]$. Since there is a natural embedding
  $D\otimes_C C(z)\hookrightarrow D(z)$, it follows that $[D(z):C(z)]\geq
  [D\otimes_CC(z):C(z)]=[D:C]$ and, hence, $D(z)$ is infinite dimensional over $C(z)$, which is
  its (infinite) centre, by \cite[Proposition~2.1.5]{pC95}. So Theorem~\ref{th:freefield}
  applies and, therefore, $D(( z))$ contains a free field $\ff{C(z)}{}{X}$
  on a countable set $X$. Finally, by \cite[Proposition~5.4.4]{pC95}, $D(( z))$
  contains a free field $\ff{C}{}{X}$.
\end{proof}



We finish this section by remarking that free fields can be ``pulled
back'' through specializations.

\begin{theorem}\label{th:spec}
  Let $D$ be a skew field with a subfield $K$ and let $E$ be a $D$-field
  such that $E$ contains a free field $\ff{D}{K}{X}$. Let $F$ be a $D$-field
  and suppose that there exists a specialization from $F$ to $E$ satisfying
  the following condition on a local homomorphism $\alpha$
  representing it,
  $$
    \alpha^{-1}(x)\cap C_F(K) \ne \emptyset \text{, for every $x\in X$,}
  $$
  where $C_F(K)$ denotes the centralizer of $K$ in $F$.
  Then $F$ contains a free field $\ff{D}{K}{X}$.
\end{theorem}

The theorem above will be obtained as a consequence of the
following more general result.

\begin{proposition}
  Let $R$ be a ring with a universal field of fractions $U$. Let $F$ and
  $E$ be $R$-fields and suppose that $U$ is an $R$-subfield of $E$.
  If there exists a specialization from $F$ to $E$ then $U$
  is isomorphic to an $R$-subfield of $F$.
\end{proposition}

\begin{proof}
  The canonical maps $R\rightarrow F$ and $R\rightarrow E$ are injective,
  because $E$ contains $U$, a field of fractions of $R$, and there exists an $R$-homomorphism
  $\alpha\colon F_0\rightarrow E$, where $F_0$ is an $R$-subring of $F$. In fact,
  there exists such an $\alpha$ with the property that any element not in $\ker\alpha$
  is invertible in $F_0$. Now, $\im\alpha\cong F_0/\ker\alpha$
  is an $R$-subfield of $E$ and, since $U$ is generated as a field by $R$, we have
  $U\subseteq \im\alpha$. Let $L$ be the subfield of $F$ generated by $R$ and set $L_0=L\cap F_0$.
  Clearly, $\alpha(L_0)\subseteq U$. So we have an $R$-map $\beta\colon L_0\rightarrow U$
  given by the restriction of $\alpha$ to $L_0$.
  $$
    
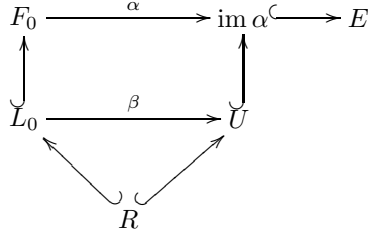
\xymatrix{
      F_0\ar[rr]^{\alpha} & & \im\alpha\ar@{^{(}->}[r]& E\\
      L_0\ar@{^{(}->}[u] \ar[rr]^{\beta} & & U\ \ar@{^{(}->}[u]\\
      & R\ar@{_{(}->}[ul] \ar@{^{(}->}[ur]}
  $$
  By definition of $\beta$, any element
  of $L_0$ not in the kernel of $\beta$ is invertible in $L_0$. Therefore,
  there exists a specialization from the $R$-field $L$ to $U$. Since $U$ is a universal
  $R$-field, there exists a specialization from $U$
  to $L$ and, hence, $U\cong L\subseteq F$, by \cite[Theorem~7.2.4]{pC95}.
\end{proof}

\begin{proof}[Proof of Theorem~\ref{th:spec}]
  Let $F_0$ denote the domain of $\alpha$. For each $x\in X$ fix $y_x\in
  \alpha^{-1}(x)\cap C_F(K)$. Let
  $\varphi\colon D_K\langle X\rangle\rightarrow F$ be the
  unique $D$-ring homomorphism such that $\varphi(x)=y_x$. Then $F$
  is a $D_K\langle X\rangle$-field. Since $\ff{D}{K}{X}$ is the universal
  field of fractions of $D_K\langle X\rangle$, it follows from the
  previous proposition that $F\supseteq\ff{D}{K}{X}$.
\end{proof}

\section{Some skew fields that do not contain free fields}\label{sec:noff}

\noindent




In this section we shall present a theorem giving a necessary condition
for a division ring to contain a free field. This will then be used
in order to produce examples of division rings which, although containing
free algebras, do not contain free fields.

Throughout this section, $k$ will denote a commutative field.

Let $A$ be a $k$-algebra. Denote by $A^{\op}$
the opposite algebra of $A$. Regard the $(A,A)$-bimodule $A$ as a left module over the
enveloping algebra $A\otimes_k A^{\op}$ in the usual way, that is, via
$a\otimes b\cdot x = axb$, for all $a,b,x\in A$. We recall that the multiplication map
$$\begin{array}{rcl}
  \mu\colon A\otimes_k A^{\op}&\longrightarrow & A\\
  a\otimes b&\longmapsto &ab
\end{array}$$
is a surjective homomorphism of left $A\otimes_k A^{\op}$-modules
and that $\ker\mu$ is generated as a left $A$-module by the set
$\{a\otimes 1-1\otimes a : a\in A\}$.




\begin{lemma}\label{le:Lichtman}
  Let $D$ be a skew field which is a $k$-algebra. If $D\otimes_k D^{op}$
  is a left noetherian algebra, then $D$ does not contain
  an infinite strictly ascending chain of subfields $D_1\subsetneqq D_2 \subsetneqq D_3
  \subsetneqq \dots$ which are $k$-subalgebras.
\end{lemma}

\begin{proof}
  Suppose that $D$ contains an infinite strictly ascending chain of subfields
  $$
    D_1\subsetneqq\ D_2\subsetneqq\ D_3\subsetneqq\dots
  $$
  which are $k$-subalgebras. We start by fixing some notation.
  Let $R=D\otimes_k D^{\op}$; for each $n\geq 1$, let $R_n=D_n\otimes_k D_n^{\op}$,
  let $\mu_n\colon R_n\rightarrow D_n$ denote the multiplication map in $D_n$
  and let $I_n=\ker\mu_n$. It is clear that $R$ contains the following
  chain of left ideals,
  $$
    RI_1\subseteq RI_2\subseteq RI_3\subseteq\dots
  $$
  We shall show that for every $n\geq 1$, we have $RI_n\neq RI_{n+1}$.
  For the sake of simplicity, we shall show that $RI_1\neq RI_2$.

  Since $I_1\subseteq I_2$ and $I_2$ is a left ideal of $R_2$, we
  have $R_2I_1\subseteq I_2$. We start by showing that $R_2I_1\neq I_2$.
  The ring $S=D_2\otimes_{D_1}D_2^{\op}$ has a left $R_2$-module structure
  such that $(a\otimes_k b)\cdot(x\otimes_{D_1}y) = ax\otimes_{D_1}yb$,
  for all $a,b,x,y\in D_2$. Consider the homomorphism of
  left $R_2$-modules $\varphi\colon R_2\rightarrow S$ such that
  $\varphi(a\otimes_k b) =a\otimes_{D_1} b$, for all $a,b\in D_2$.
  Given $z\in D_1$, we have $\varphi (1\otimes_k z-z\otimes_k 1) =
  1\otimes_{D_1}z-z\otimes_{D_1} 1 =0$. It follows that $R_2I_1\subseteq\ker\varphi$,
  for the set $\{1\otimes_k z-z\otimes_k 1:z\in D_1\}$ generates $I_1$ as a
  left $R_1$-module and, hence, is a generating set for the left ideal
  $R_2I_1$ of $R_2$. Suppose that $R_2I_1=I_2$ and take $x\in D_2\setminus D_1$. Then
  $1\otimes_k x-x\otimes_k 1\in I_2=R_2I_1\subseteq\ker\varphi$. This implies
  $$
    1\otimes_{D_1} x = \varphi(1\otimes_k x)=\varphi(x\otimes_k 1) =x\otimes_{D_1}1,
  $$
  which is a contradiction with the choice of $x$. So $R_2I_1\neq I_2$.

  To prove that $RI_1\subsetneqq RI_2$, note that since $D_2$ is
  a free $D_1$-module, $R_2$ is a free right $R_1$-module and, thus,
  it is a faithfully flat right $R_1$-module. Similarly, $R$ is a
  faithfully flat right $R_2$-module. We then have
  $$
    RI_1 = R\otimes_{R_2}(R_2\otimes_{R_1}I_1) =
    R\otimes_{R_2}R_2I_1\subsetneqq R\otimes_{R_2}I_2= RI_2,
  $$
  which is what we had to prove.
\end{proof}

As a consequence we obtain the following necessary condition
for a skew field to contain a free field.

\begin{theorem}\label{th:nofreefield}
  Let $D$ be a skew field and let $k$ be a central subfield of $D$.
  If $D$ contains a free field $\ff{k}{}{x,y}$ then $D\otimes_k D^{\op}$
  cannot be a left noetherian ring.
\end{theorem}

\begin{proof}
  If $D$ contains the free field $\ff{k}{}{x,y}$ then, by \cite[Corollary
  5.5.9]{pC95}, $D$ contains a free
  field with a countable number of generators, say,
  $\ff{k}{}{x_1,x_2,\dots}\subseteq D$. For every $n\geq 1$, letting
  $D_n=\ff{k}{}{x_1,\dots,x_n}$, we obtain an infinite strictly ascending
  chain of subfields in $D$, $D_1\subsetneqq D_2\subsetneqq\dots$.
  By Lemma~\ref{le:Lichtman}, $D\otimes_k D^{\op}$ is not a left
  noetherian ring.
\end{proof}



In \cite {lM83}, Makar-Limanov showed that the field of fractions of
the first Weyl algebra over a commutative field $k$ of characteristic zero
contains a noncommutative free algebra over $k$ and, consequently, it contains a field of
fractions of the free algebra. However, this field of fractions
is not a free field, as the following consequence of Theorem~\ref{th:nofreefield}
shows.

\begin{corollary}
  The field of fractions of the first Weyl algebra over a
  commutative field does not contain a noncommutative free field (over any
  subfield).
\end{corollary}

\begin{proof}
  Let $A=k\langle x,y : xy-yx =1 \rangle$ be the first Weyl algebra over a field $k$
  of characteristic zero,
  let $D$ be its left classic field of fractions and let
  $S=A\otimes_k D^{\op}$. Then $S$ is a Weyl algebra over the
  skew field $D^{\op}$, that is, $S\cong D^{\op}\langle x, y :xy-yx=1
  \rangle$. Since $A$ is a left Ore domain, $T=(A\setminus\{0\})\otimes_k 1$ is a left denominator
  subset of $S$. Since $S$ is left noetherian, because it is a Weyl algebra over a skew field,
  $T^{-1}S\cong D\otimes_k D^{\op}$ is also left noetherian. By Theorem~\ref{th:nofreefield},
  $D$ does not contain a free field over $k$. Therefore, $D$ does not contain a free field over
  any central subfield, by \cite[Lemma~9]{kC03}. Consequently, $D$ does not
  contain a free field over any subfield, by \cite[Corollary~pg.~114]{pC77} and
  \cite[Theorem~5.8.12]{pC95}.
\end{proof}



Now consider a nonabelian torsion-free polycyclic-by-finite group
$G$. We recall that the group ring $KG$ is a noetherian domain for
every skew field $K$ (cf.~\cite[Proposition~1.6 and
Corollary~37.11]{dP89}). Let $k$ be a commutative field and consider
the group algebra $kG$. Let $F$ denote the left classic field of
fractions of $kG$ and consider the $k$-algebra $R=kG\otimes_k
F^{\op}\cong F^{\op}G$. Proceeding in similar manner as we did
above it is possible to show that $F\otimes_k F^{\op}$ is a left
noetherian ring. By Theorem~\ref{th:nofreefield} we have the
following result.

\begin{corollary}
  The field of fractions of the group algebra of a nonabelian
  torsion-free polycyclic-by-finite group over a commutative field
  does not contain a noncommutative free field (over any subfield).\hfill\qed
\end{corollary}

Let $H$ be a nonabelian torsion-free finitely generated nilpotent
group and consider the group algebra $kH$. Since $H$ is
polycyclic-by-finite then the group algebra $kH$ is a noetherian
domain. We remark that, again, Makar-Limanov proved in \cite{lM84}
that the field of fractions $Q$ of $kH$ contains a noncommutative
free algebra over $k$. By what we have just seen, the subfield of
$Q$ generated by this free algebra is not a free field.


\section{Free fields in Malcev-Neumann series ring}\label{sec:mn}


\noindent

In this section we shall see that a ring of Malcev-Neumann series of
a torsion-free nilpotent group over a division ring contains a free
field.

Let $G$ be a finitely generated torsion-free nilpotent group. We follow
\cite{sJ55} and recall some facts about $G$. First, $G$ has at least one
central series
\begin{equation}
  G=F_1 \supseteq F_2 \supseteq F_3 \supseteq \dots \supseteq F_r \supseteq
  F_{r+1}=\{1\} \label{eq:F-series}
\end{equation}
such that for all $i=1, \dots, r$, $F_i/F_{i+1}$ is an infinite cyclic
group. (Such a series can be obtained as a refinement of the upper
central series of $G$, for instance.) A central series of $G$ whose factors are
infinite cyclic will be called an \emph{$\mathcal F$-series}. Let $f_i$ be a
representative in $G$ of a generating element of $F_i/F_{i+1}$. It follows that
every element of $G$ can be written uniquely in the form
\begin{equation}\label{eq:normalform}
  f_1^{\alpha_1}f_2^{\alpha_2} \dots f_r^{\alpha_r},
\end{equation}
with $\alpha_1, \alpha_2, \dots, \alpha_r\in \mathbf{Z}$. Given
subgroups $X,Y$ of $G$ the subgroup generated by all commutators
$(x,y)$ with $x\in X$ and $y\in Y$ will be denoted by $(X,Y)$. So,
since \eqref{eq:F-series} is a central series for $G$, we have
$(F_i,G)\subseteq F_{i+1}$, for all $i=1, \dots, r$, and, hence
$(F_j, F_i)\subseteq F_{j+1}$, whenever $j\geq i$. Because $F_{j+1}$
is a finitely generated torsion-free nilpotent group with $\mathcal
F$-series $F_{j+1}\supseteq F_{j+2}\supseteq \dots \supseteq
F_r\supseteq F_{r+1}=\{1\}$, given $\alpha_i,\alpha_j\in\mathbf{Z}$,
we have
$$
  f_{j}^{\alpha_{j}}f_{i}^{\alpha_{i}}=
  f_i^{\alpha_{i}}f_{j}^{\alpha_{j}}f_{j+1}^{\gamma_{j+1}}f_{j+2}^{\gamma_{j+2}}\dots f_r^{\gamma_r},
$$
for some $\gamma_{j+1},\dots,\gamma_{r}\in\mathbf{Z}$.
In particular, $F_r=\langle f_r \rangle \subseteq
\mathcal{Z}(G)$, the centre of $G$. Finally, $G$ is an ordered group
with the lexicographic ordering, that is,
$$
  f_1^{\alpha_1}f_2^{\alpha_2}\dots f_r^{\alpha_r} < f_1^{\beta_1}f_2^{\beta_2}\dots f_r^{\beta_r}
$$
if and only if there exists some $s$ for which
$$
  \alpha_{s}<\beta_{s} \text{, while } \alpha_i=\beta_i \text {, for all $i=1,\dots, s-1$.}
$$

Assume that $G$ is nonabelian, that the elements $f_1,\dots,
f_r$ in \eqref{eq:normalform} have been chosen and that the order in
$G$ is defined as above. Let $K$ be a skew field and consider the
Malcev-Neumann skew series field $K(( G,\sigma))$, where
$\sigma\colon G\rightarrow \Aut(K)$ is a group homomorphism of $G$
into the automorphism group $\Aut(K)$ of $K$. Recall that $K(( G,
\sigma))$ is constituted by those formal series $f=\sum_{g\in G}
a_gg$, with $a_g\in K$, whose support $\supp(f)=\{g\in G : a_g\neq
0\}$ is a well-ordered subset of $G$. The product in $K(( G,
\sigma))$ is induced by the relations $ga=\sigma(g)(a)g$, for all
$g\in G$ and $a\in K$. When $\sigma$ is the trivial homomorphism,
sending every element of $G$ to the identity map of $K$, we shall
simplify the notation and write $K(( G))$ for the corresponding
Malcev-Neumann skew field. The subring of $K(( G, \sigma))$ of
elements of finite support is the usual skew group ring of $G$ over
$K$, henceforth denoted by $K(G, \sigma)$. When $\sigma$ is trivial
this subring is the usual group ring and is denoted by $KG$. By
\cite[Corollary~37.11]{dP89}, $K(G, \sigma)$ is an Ore domain with
field of fractions $Q(K(G, \sigma))$; so we have
$K(G,\sigma)\subseteq Q(K(G,\sigma))\subseteq K(( G,\sigma))$.

We shall show that there exists a natural discrete valuation on $K(G,\sigma)$
which extends to $K(( G,\sigma))$ turning this Malcev-Neumman skew field
into a complete field. This will eventually yield an embedding of a free
field into it.

Since any element of $G$ can be written uniquely in the form
\eqref{eq:normalform}, an element $x$ of $K(G,\sigma)$ can be
written uniquely as a finite sum
$$
  x=\sum_{I} a_{I}f_1^{\alpha_1}f_2^{\alpha_2}\dots f_r^{\alpha_r},
$$
where $I=(\alpha_1,\alpha_2,\dots, \alpha_r)\in \mathbf{Z}^r$ and $a_{I}\in
K$.

\begin{lemma}\label{le:valuesgp}
  The map $o\colon K(G,\sigma)\rightarrow \mathbf{Z}\cup\{\infty\}$, given by
  $$
    o(x) =  \min\{\alpha_1 : f_1^{\alpha_1}f_2^{\alpha_2}\dots f_r^{\alpha_r} \in
    \supp(x) \text{ for some } \alpha_2, \dots, \alpha_r \in\mathbf{Z}\},
  $$
  if $x\ne 0$ and $o(0)=\infty$, is a valuation function on $K(G,\sigma)$.
\end{lemma}
\begin{proof} Take nonzero elements $x,y\in K(G,\sigma)$ with $o(x)=\alpha$ and
  $o(y)=\beta$, say $\min\supp(x) = f_1^{\alpha}f_2^{\alpha_2}\dots f_r^{\alpha_r}$
  and $\min\supp(y) = f_1^{\beta}f_2^{\beta_2}\dots f_r^{\beta_r}$, and suppose that the
  coefficients of these elements of $G$ occurring in $x$ and $y$ are $a$ and $b$, respectively.
  Then the element
  $$
    g=(f_1^{\alpha}f_2^{\alpha_2}\dots f_r^{\alpha_r})(f_1^{\beta}f_2^{\beta_2}\dots f_r^{\beta_r}) =
    f_1^{\alpha+\beta}f_2^{\alpha_2+\beta_2}f_3^{\gamma_3}\dots f_r^{\gamma_r},
  $$
  for appropriate $\gamma_3,\dots,\gamma_r\in\mathbf{Z}$, occurs in $xy$ with coefficient
  $a\sigma(f_1^{\alpha}f_2^{\alpha_2}\dots f_r^{\alpha_r})(b)\neq 0$. So $g$ is the least
  element in $\supp(xy)$ and, hence, $o(xy) = \alpha + \beta$.
  The other conditions required for $o$ to be a valuation are clearly satisfied.
\end{proof}

Denote by $\nu\colon Q(K(G, \sigma)) \rightarrow \mathbf{Z}\cup
\{\infty\}$ the unique extension of $o$ to a valuation on
$Q(K(G, \sigma))$.

Given a nonzero $x=\sum_I a_I f_1^{\alpha_1}f_2^{\alpha_2}\dots f_r^{\alpha_r}\in K(( G, \sigma))$,
since $\supp(x)$ is well-ordered, the integer
$\min\{\alpha_1 : f_1^{\alpha_1}f_2^{\alpha_2}\dots f_r^{\alpha_r} \in
\supp(x) \text{ for some } \alpha_2, \dots, \alpha_r \in\mathbf{Z}\}$
is well defined.

\begin{lemma}\label{le:valuemns}
  The map $\hat{\nu}\colon K(( G, \sigma)) \rightarrow \mathbf{Z}\cup \{\infty\}$,
  defined by
  $$
    \hat{\nu}(x) =  \min\{\alpha_1 : f_1^{\alpha_1}f_2^{\alpha_2}\dots f_r^{\alpha_r} \in
    \supp(x) \text{ for some } \alpha_2, \dots, \alpha_r \in\mathbf{Z}\},
  $$
  if $x\ne 0$ and $\hat{\nu}(0)=\infty$, is a valuation function on $K((G, \sigma))$ which
  extends the valuation $\nu$ on $Q(K(G, \sigma))$.
\end{lemma}

\begin{proof}
  The same argument in the proof of Lemma \ref{le:valuesgp}
  shows that $\hat{\nu}$ is a valuation function. By definition, $\hat{\nu}|_{K(G,\sigma)} = o$, so
  $\hat{\nu}|_{Q(K(G,\sigma))} = \nu$.
\end{proof}

The next step is to show that $K((G,\sigma))$ is complete with respect
to this valuation.

\begin{lemma}\label{le:completemns}
  The skew field of Malcev-Neumann series $K((G,\sigma))$ is
  complete with respect to the valuation $\hat{\nu}$.
\end{lemma}

\begin{proof}
  Take a real constant $c\in(0,1)$ and consider the metric $d$ in $K((G,\sigma))$
  induced by $\hat{\nu}$, that is, we set $d(x,y)=c^{\hat{\nu}(x-y)}$, for $x,y\in
  K((G,\sigma))$. Let $(u_n)_{n\in\mathbf{N}}$ be a Cauchy sequence in $K(( G, \sigma))$, say
  $$
    u_n=\sum_{I=(\alpha_1,\dots,  \alpha_r)}
      a_{I}^{(n)}f_1^{\alpha_1}f_2^{\alpha_2}\dots f_r^{\alpha_r},
  $$
  where $a_{I}^{(n)} \in K$. For each $k\in \mathbf{Z}$,
  choose $n_k \in \mathbf{N}$ such that $d(u_p,u_q)<c^k$,
  whenever $p, q \geq n_k$, satisfying $n_0<n_1<n_2<\dots<n_k<\dots$ and
  consider the subsequence $(u_{n_k})_{k\in\mathbf{N}}$ of $(u_n)$.
  Clearly, for every $k\in\mathbf{Z}$, we have $d(u_p,u_q)<c^k$ if and only if
  $a_{I}^{(p)}=a_{I}^{(q)}$, for all $I=(\alpha_1,\dots, \alpha_r)$ with $\alpha_1\leq k$.
  So, if $l\geq k\geq 0$, then
  \begin{align}
    a_I^{(n_k)}&=a_I^{(n_0)} \text{, for all $I=(\alpha_1,\dots,\alpha_r)$
      with $\alpha_1\leq 0$}\label{MN2}\\
    \intertext{and}
    a_I^{(n_l)}&=a_I^{(n_k)} \text{, for all $I=(\alpha_1,\dots,\alpha_r)$
      with $\alpha_1\leq k$.}\label{MN3}
  \end{align}
  We shall show that $(u_{n_k})$ is convergent in $K((G,\sigma))$. Consider
  the element
  \begin{equation}\label{eq:sumu}
    u=\sum_{\substack{I=(\alpha_1,\dots,\alpha_r)\\\alpha_1<0}}
        a_I^{(n_0)}f_1^{\alpha_1}f_2^{\alpha_2}\dots f_r^{\alpha_r} +
      \sum_{\substack{I=(\alpha_1,\dots,\alpha_r)\\\alpha_1\geq0}}
        a_{I}^{(n_{\alpha_1})}f_1^{\alpha_1}f_2^{\alpha_2}\dots f_r^{\alpha_r}.
  \end{equation}
  Let us prove that $u\in K(( G, \sigma))$, that is, that $\supp(u)$ is well-ordered.
  Let $v$ denote the first part of the sum in $u$ (that is, the sum of all the
  terms with $\alpha_1<0$ in \eqref{eq:sumu}) and $w$, the second.
  Because it is a subset of $\supp(u_{n_0})$, $\supp(v)$ is well-ordered.
  As for $w$, let $S$ be a nonempty subset of $\supp(w)$ and let $\alpha$ be the
  smallest integer such that $f_1^{\alpha}f_2^{\alpha_2}\dots f_r^{\alpha_r}\in S$
  for some choice of $\alpha_2,\dots,\alpha_r\in\mathbf{Z}$. Such an $\alpha$ exists because,
  by definition, if $f_1^{\alpha_1}\dots f_r^{\alpha_r}\in\supp(w)$, then,
  necessarily, $\alpha_1\geq 0$. Let $T$ be the set formed by all the elements
  of $S$ which are of the form $f_1^{\alpha}f_2^{\alpha_2}\dots f_r^{\alpha_r}$, for
  some $\alpha_2,\dots,\alpha_r\in\mathbf{Z}$. We have seen that $T\ne \emptyset$.
  Furthermore, by the definition of $w$, we have $T\subseteq\supp(u_{n_{\alpha}})$ and,
  therefore, $T$ has a smallest element which is easily seen to be the
  smallest element of $S$. It follows that $\supp(w)$ is well-ordered and,
  therefore, $u=v+w\in K((G,\sigma))$.

  Finally, we prove that $(u_{n_k})$ converges to $u$. Given $\epsilon>0$
  let $t\in \mathbf{Z}$ be such that $c^t<\epsilon$. For every $k\geq t$, we have
  \begin{equation*}\begin{split}
    u_{n_k}-u
      & = \sum_{\substack{I=(\alpha_1,\dots,\alpha_r)\\\alpha_1<0}}
          (a_{I}^{(n_k)}-a_I^{(n_0)})f_1^{\alpha_1}f_2^{\alpha_2}\dots f_r^{\alpha_r}\\
      & \quad +\sum_{\substack{I=(\alpha_1,\dots,\alpha_r)\\ 0\leq\alpha_1\leq k}}
          (a_{I}^{(n_k)}-a_I^{(n_{\alpha_1})})f_1^{\alpha_1}f_2^{\alpha_2}\dots f_r^{\alpha_r}\\
      & \quad +\sum_{\substack{I=(\alpha_1,\dots,\alpha_r)\\\alpha_1\geq k+1}}
          (a_{I}^{(n_k)}-a_I^{(n_{\alpha_1})})f_1^{\alpha_1}f_2^{\alpha_2}\dots f_r^{\alpha_r}\\
      & = \sum_{\substack{I=(\alpha_1,\dots,\alpha_r)\\\alpha_1\geq k+1}}
          (a_{I}^{(n_k)}-a_I^{(n_{\alpha_1})})f_1^{\alpha_1}f_2^{\alpha_2}\dots f_r^{\alpha_r},
  \end{split}\end{equation*}
  by \eqref{MN2} and \eqref{MN3}. Thus, $\hat{\nu}(u_{n_k}-u)\geq
  k+1>k$ and hence
  $ \hat{d}(u_{n_k},u)=c^{\hat{\nu}(u_{n_k}-u)}<c^k<\epsilon$.
  Therefore, the subsequence $(u_{n_k})$ of $(u_n)$ converges to $u$.
  Hence $(u_n)$ is convergent. It follows that $K(( G, \sigma))$ is complete.
\end{proof}



In order to prove the existence of a free field in $K((G,\sigma))$
using Theorem~\ref{th:freefield} we need to establish the infinite
dimensionality of $Q(K(G,\sigma))$ over its centre. We state a
result that holds in a larger context.

\begin{lemma}\label{le:ffskg}
  Let $H$ be a nonabelian ordered group, let $K$ be a skew field and let
  $\tau\colon H\rightarrow \Aut(K)$ be a group homomorphism. Then
  any field of fractions of the skew group ring $K(H,\tau)$ is
  infinite-dimensional over its centre.
\end{lemma}

\begin{proof}
  Let $F$ be a field of fractions of $K(H,\tau)$ and denote by $\pi$ the
  prime subfield of $K$. If $F$ were finite-dimensional
  over its centre, then $F$ and, therefore, the group algebra $\pi H$ would be PI-rings.
  By \cite[Lemma~10]{kC03}, this would imply that $G$ is an abelian group.
\end{proof}

\begin{theorem}\label{th:MNS}
  Let $K$ be a skew field, let $G$ be a nonabelian finitely generated
  torsion-free nilpotent group and let $\sigma\colon G\rightarrow \Aut(K)$ be
  a group homomorphism. If the centre $Z$ of the classical field of fractions
  of the skew group ring $K(G,\sigma)$ is infinite then the skew field of
  Malcev-Neumman series $K(( G, \sigma))$ contains a free field $\ff{Z}{}{X}$
  on a countable set $X$.
\end{theorem}

\begin{proof}
  Let $Q$ denote the classical field of fractions of $K(G,\sigma)$.
  Since by Lemma~\ref{le:ffskg} $Q$ is infinite-dimensional
  over its centre $Z$, it follows from Theorem~\ref{th:freefield} that
  the completion $\widehat{Q}$ of $Q$ with respect to the topology defined by the
  discrete valuation extending the valuation $o$ of $K(G,\sigma)$
  defined in Lemma~\ref{le:valuesgp} contains a free field $\ff{Z}{}{X}$
  on a countable set $X$. As a consequence of Lemma~\ref{le:valuemns}, the natural
  monomorphism $Q\hookrightarrow K(( G, \sigma))$ is an isometric embedding.
  Since, by Lemma~\ref{le:completemns}, $K(( G, \sigma))$ is a complete
  metric space, the closure of $Q$ in $K((G,\sigma))$ coincides with
  $\widehat{Q}$. It thus follows that $K(( G, \sigma))$
  contains $\ff{Z}{}{X}$.
\end{proof}

If $\sigma$ is trivial, then the centre $Z$ of $Q$ is necessarily
infinite, because, using the notation in the first paragraph of this
section, the infinite cyclic group $F_r$ is a subgroup of the centre
$\mathcal{Z}(G)$ of $G$. Therefore,
$F_r\subseteq\mathcal{Z}(G)\subseteq Z(KG)\subseteq Z$. This proves
the next result.

\begin{corollary}
  Let $K$ be a skew field and let $G$ be a nonabelian finitely generated
  torsion-free nilpotent group. Then the skew field of
  Malcev-Neumman series $K((G))$ contains a free field $\ff{Z}{}{X}$ over
  the centre $Z$ of the classical field of fractions of the group ring
  $KG$ on a countable set $X$.\hfill\qed
\end{corollary}

%

\section{Skew Laurent series}\label{sec:ls}

In this section we show that a large class of division rings, namely
fields of fractions of skew polynomial rings over skew fields, are
endowed with natural valuations whose completions contain free
fields. In particular, it will be shown that the Weyl field can be
embedded in a complete skew field containing a free field. Similar
methods will be applied to deal with fields generated by quantum
matrices.



Let $A$ be a ring, let $\sigma$ be an endomorphism of $A$ and let
$\delta$ be a (right) $\sigma$-derivation of $A$, that is,
$\delta\colon A\rightarrow A$ is an additive map satisfying
$$
\delta(ab) = \delta(a)\sigma(b)+a\delta(b),
$$
for all $a,b\in A$. The \emph{skew polynomial ring in $x$ over $A$},
hereafter denoted by $A[x;\sigma,\delta]$, is the free right
$A$-module on the nonnegative powers of $x$ with multiplication
induced by
$$
ax = x\sigma(a)+\delta(a),
$$
for all $a\in A$. If $A$ is a skew field then $A[x; \sigma, \delta]$
is a principal right ideal domain and its field of fractions will be
denoted by $A(x; \sigma, \delta)$. When $A$ is a right Ore domain
with field of fractions $Q(A)$ and $\sigma$ is injective, $\sigma$
and $\delta$ can be extended to $Q(A)$ and we can consider the skew
polynomial ring $Q(A)[x; \sigma, \delta]$. In this case we have
$$
A[x; \sigma,\delta] \subseteq Q(A)[x; \sigma, \delta] \subseteq
Q(A)(x; \sigma,\delta)
$$
and so $A[x; \sigma, \delta]$ is a right Ore domain with field of
fractions $Q(A)(x; \sigma, \delta)$.

For an arbitrary endomorphism $\sigma$ of the skew field $K$ a
further construction will be considered. The ring of \emph{skew
power series} $K[[ x;\sigma]]$ is defined to be the set of power
series on $x$ of the form $\sum_{i=0}^{\infty} x^ia_i$, with $a_i\in
K$, where addition is done coefficientwise and multiplication is
given by the rule $ax=x\sigma(a)$, for $a\in K$. When $\sigma$ is an
automorphism, we can embed $K[[x;\sigma]]$ into the ring of
\emph{skew Laurent series} $K((x;\sigma))$ whose elements are series
$\sum_{-r}^{\infty}x^ia_i$, with $r$ a nonnegative integer. In
$K((x;\sigma))$, multiplication satisfies $ax^n=x^n\sigma^n(a)$, for
every integer $n$ and $a\in K$. It is well known that
$K((x;\sigma))$ is a skew field containing $K(x;\sigma)$.

As mentioned in \cite[Section~2.3]{pC95}, when $\delta\neq 0$ we
must look at a more general construction. Given an automorphism
$\sigma$ of a skew field $K$ and a $\sigma$-derivation $\delta$ on
$K$, denote by $R$ the ring of all power series
$\sum_{i=0}^{\infty}a_iy^i$ on $y$ with multiplication induced by
$ya = \sum_{i=0}^{\infty}\sigma\delta^i(a)y^{i+1}$, for $a\in K$. In
\cite[Theorem 2.3.1]{pC95} it is shown that $R$ is a domain and
$S=\{1, y, y^2,\dots\}$ is a left Ore set in $R$ whose localization
$S^{-1}R$ is a skew field, consisting of all skew Laurent series
$\sum_{i=r}^{\infty} a_iy^i$, with $r\in\mathbb{Z}$. Because in
$S^{-1}R$ we have $ya=\sigma(a)y+y\delta(a)y$, for all $a\in K$, it
follows that $ay^{-1}=y^{-1}\sigma(a)+\delta(a)$. Hence, the subring
generated by $y^{-1}$ in $S^{-1}R$ is an ordinary skew polynomial
ring on $y^{-1}$. Because of that we shall adopt the notation
$R=K[[x^{-1}; \sigma, \delta]]$ and $S^{-1}R=K((x^{-1}; \sigma,
\delta))$. So we have an embedding $K[x;\sigma,\delta]\subseteq K((
x^{-1}; \sigma, \delta))$, sending $x$ to $y^{-1}$ and, hence,
$K(x;\sigma,\delta)\subseteq K((x^{-1};\sigma,\delta))$. It is also
clear that $K((x^{-1};\sigma,\delta))=K[[x^{-1}; \sigma,
\delta]]+K[x;\sigma, \delta]$.

We now define a valuation function on $K((x^{-1};\sigma,\delta))$.

\begin{lemma} \label{le:valLaurent}
  Let $K$ be a skew field with an automorphism $\sigma$ and a $\sigma$-derivation
  $\delta$. Then the map $\omega\colon K((x^{-1};\sigma,\delta)\rightarrow\mathbb{Z}\cup \{\infty\}$
  defined by
  $$
  \omega(f)=\sup\{n : f\in K[[x^{-1}; \sigma, \delta]] y^n\},
  $$
  is a valuation on $K((x^{-1};\sigma,\delta))$.
\end{lemma}

\begin{proof}
  The only nonobvious fact to be proved is that if $f$ and $g$ are
  nonzero elements of $K((x^{-1};\sigma,\delta))$, then $\omega(fg) = \omega(f)+
  \omega(g)$. This follows from the fact that, since $\sigma$ is
  invertible, it is possible to write, for
every integer $n$ and $a\in K$, $y^na=\sigma^n(a)y^n+hy^{n+1}$, for
some $h\in K[[x^{-1}; \sigma, \delta]]$.
\end{proof}


The field $K(x;\sigma,\delta)$ has a valuation ``at infinity'' $\nu$
which extends the valuation $-\deg$ on $K[x; \sigma, \delta]$, where $\deg$
denotes the usual degree function on $K[x;\sigma,\delta]$, that is,
$\deg(\sum x^ia_i) = \max\{i : a_i\ne 0\}$. Clearly,
$\omega|_{K(x; \sigma, \delta)}=\nu$. The following lemma shows
that $K((x^{-1};\sigma,\delta))$ is the completion of $K(x;\sigma,\delta)$.

\begin{lemma}\label{le:completepol}
  Let $K$ be a skew field with an automorphism $\sigma$ and a $\sigma$-derivation
  $\delta$. With respect to the valuation $\nu$ defined above,
  $K((x^{-1};\sigma,\delta))$ is the completion of $K(x;\sigma,\delta)$.
\end{lemma}

\begin{proof}
  That $K(x;\sigma,\delta)$ is dense in $K((x^{-1};\sigma,\delta))$
  follows from the fact that every Laurent series $\sum_{i=n}^{\infty}a_iy^i$
  in $K((x^{-1};\sigma,\delta))$ is the limit of its partial sums, which are
  elements in $K(x;\sigma,\delta)$.
  It remains to prove that $K((x^{-1}; \sigma, \delta))$ is complete. Let
  $(u_n)$ be a Cauchy sequence in $K((x^{-1}; \sigma, \delta))$, say
  $$
  u_n=\sum_{i\geq m_n}a_i^{(n)}y^i,
  $$
  with $a_i^{(n)}\in K$ and $m_n\in\mathbf{Z}$. We proceed as in the proof
  of Lemma~\ref{le:completemns}. For each $k\in \mathbf{Z}$ choose nonnegative
  integers $n_k\in \mathbf{N}$ satisfying $n_0 < n_1 < n_2 < \dots$
  such that $\omega(u_p-u_q)>k$, for all $p,q\geq n_k$. Consider the following
  element of $K((x^{-1};\sigma,\delta))$,
  $$
  u=\sum_{i<0}a_i^{(n_0)}y^i+\sum_{i\geq 0}a_i^{(n_i)}y^i.
  $$
  We shall prove that $(u_n)$ converges to $u$. Indeed, for all $k$,
  $\omega(u_{n_k}-u)\geq k+1$ and, hence, the subsequence $(u_{n_k})$
  of $(u_n)$ converges to $u$. Since $(u_n)$ is a Cauchy sequence,
  $(u_n)$ converges to $u$. It follows that $K((x^{-1};\sigma,\delta))$
  is complete.
\end{proof}

In the case that $\delta=0$, $K(x; \sigma)$ has an $x$-adic valuation
$\eta$, defined by $\eta(fg^{-1})=o(f)-o(g)$, for all $f,g\in K[x;\sigma]$, $g\ne 0$,
where $o(\sum x^ia_i)=\min\{i : a_i\neq 0\}$ is the order function on $K[x;\sigma]$,
which is a valuation. Denote by $\zeta$ the integer valued function on $K((x;\sigma))$
defined by $\zeta(h)=\sup\{n : h\in x^n K[[ x;\sigma]]\}$, for all $h\in K((x;\sigma))$,
$h\ne 0$. It is easy to see that $\zeta$ is a valuation on $K((x;\sigma))$ which
coincides with $\eta$ on $K(x;\sigma)$. With a proof similar to the one of
Lemma~\ref{le:completepol} it can be shown that $K(( x; \sigma))$ is the completion of
$K(x; \sigma)$ with respect to $\eta$.

Our next aim is to guarantee the existence of free fields
in completions of skew rational function fields. For that,
we recall the following definition. An automorphism $\sigma$
in a ring is said to have \emph{inner order} $r$ if $\sigma^r$ is
the least positive power of $\sigma$ which is inner. If $\sigma^r$
it is not inner for any $r>0$, then $\sigma$ is
said to have infinite inner order. We shall need the following
result.

\begin{lemma}[{\cite[Theorem~2.2.10]{pC95}}]\label{le:Laurentcentre}
  Let $K$ be a skew field with centre $C$, let $\sigma$ be an automorphism
  of $D$ and let $\delta$ be a $\sigma$-derivation. If $K[x;
  \sigma, \delta]$ is simple or $\sigma$ has infinite inner order,
  then $Z(K(x; \sigma, \delta))=\{c\in C : \sigma(c)=c
  \text{ and } \delta(c)=0\}$.\hfill\qed
\end{lemma}

We have the following result.

\begin{theorem}\label{th:ffpol}
  Let $K$ be a skew field with centre $C$, let $\sigma$ be an automorphism
  of $K$ and let $\delta$ be a $\sigma$-derivation of $K$. Suppose that
  $K[x; \sigma, \delta]$ is simple or that $\sigma$ has infinite inner order.
  If $\delta\ne 0$, let $R=K((x^{-1}; \sigma, \delta))$, otherwise, let
  $R=K(( x; \sigma))$. If the field $C_0=\{c\in C : \sigma(c)=c
  \text{ and } \delta(c)=0\}$ is infinite, then $R$ contains a free field
  $\ff{C_0}{}{X}$ on a countable set $X$.
\end{theorem}

\begin{proof}
  Since, by Lemma~\ref{le:Laurentcentre}, $Z(K(x; \sigma, \delta))=C_0$
  and $\{x^i : i\in\mathbf{Z}\}$ is linearly independent over $C_0$,
  because it is over $K$, it follows that $K(x;\sigma,\delta)$ is
  infinite dimensional over its infinite centre $C_0$. By
  Theorem \ref{th:freefield}, $R$ contains a free field on
  a countable set over $C_0$, for, as we have seen above,
  it is the completion of $K(x; \sigma, \delta)$ with respect
  to appropriate discrete valuations.
\end{proof}

We remark that the case $\delta=0$ in the above theorem had already
been considered in \cite[Theorem~4]{kC03} for a commutative field $K$
of coefficients.




An example of a skew polynomial ring with a nonzero derivation is the
first Weyl algebra over a skew field $K$,
$$
A_1(K)=\fr{K}{}{x_1,x_2 : x_1x_2-x_2x_1=1}=K[x_1]\left[x_2;I,\frac{d}{dx_1}\right],
$$
where $I$ above stands for the identity automorphism of $K[x_1]$.
It is well known that $A_1(K)$ is an Ore domain with field of
fractions $Q_1(A)=K(x_1)(x_2;I,\frac{d}{dx_1})$, called the Weyl field.

\begin{corollary}
  Let $K$ be a skew field with centre $C$. If $\charac K=0$ then
  the field of skew Laurent series $K(x_1)((x_2^{-1};I,\frac{d}{dx_1}))$ contains a free
  field $\ff{C}{}{X}$, where $X$ is a countable set.
\end{corollary}

\begin{proof}
  Let $B=K(x_1)[x_2;I,\frac{d}{dx}]$. Then $A_1(K)\subseteq B\subseteq
  K(x_1)(x_2; I,\frac{d}{dx_1})$ and $Q(B)=Q_1(A)=K(x_1)(x_2;I,\frac{d}{dx_1})$.
  Clearly, $\frac{d}{dx_1}$ is not an inner
  derivation. It follows from \cite[Corollary 3.16]{tL91} that $B$ is
  a simple ring. Since $\charac K=0$, we have $Z(K(x_1))_0=\{f\in
  Z(K(x_1)) : \frac{d}{dx_1}(f)=0\}=C(x_1)_0=C$, an infinite field.
  By Theorem \ref{th:ffpol}, the field of
  skew Laurent series $K(x_1)((x_2^{-1}; I,\frac{d}{dx_1}))$
  contains a free field on a countable set over $C$.
\end{proof}



Our next example is a division ring generated by quantum matrices.
Let $k$ be a commutative field and let $q\in k$ be a nonzero element.
The algebra $M_q(2)$ of quantum $2\times 2$ matrices
over $k$ is defined to be the $k$-algebra with four generators $a, b, c, d$ subjected
to the following relations:
\begin{align*}
  ab&=qba, &  ac&=qca, &  bc&=cb,\\
  bd&=qdb,  &  cd&=qdc, &  ad-da&=(q-q^{-1})cb.
\end{align*}
It can also be regarded as an iterated skew polynomial ring $k[a][b;
\sigma_b][c; \sigma_c][d; \sigma_d, \delta_d]$, where $\sigma_b$, $\sigma_c$ and
$\sigma_d$ are the
$k$-automorphisms satisfying
\begin{align*}
  &\sigma_b(a)=qa,\\
  &\sigma_c(b)=b, \quad \sigma_c(a)=qa,\\
  &\sigma_d(c)=qc, \quad \sigma_d(b)=qb, \quad \sigma_d(a)=a,
\end{align*}
and $\delta_d$ is the $\sigma_d$-derivation such that
$$
  \delta_d(a) = (q-q^{-1})cb,\quad \delta_d(b)=0,\quad \delta_d(c)=0.
$$
(See, for instance, \cite{cK95}.) It follows that $M_q(2)$ is a noetherian
domain with field of fractions $Q_q(2)$. For simplicity we shall write $R_a=k[a]$, $R_b=R_a[b;
\sigma_b]$, $R_c=R_b[c; \sigma_c]$. The fact that the field of fractions $Q_q(2)$ of
$M_q(2)=R_c[d;\sigma_d,\delta_d]$ contains a free algebra follows from the
following theorem.

\begin{theorem}[{\cite[Theorem]{mL86}}]\label{le:theoremlorenz}
  Let $K=k(t)$ be the field of rational functions over a commutative
  field $k$ and let $\sigma$ a $k$-automorphism of $K$ of infinite
  order. Then $K(x; \sigma)$ contains a noncommutative free
  $k$-subalgebra.\hfill\qed
\end{theorem}

\begin{corollary}
  Let $k$ be a commutative field and let $q\in k^{\times}$. If $q$ is not a root of 1, then
  the field of fractions of the $k$-algebra of $2\times 2$ quantum matrices contains a
  noncommutative free algebra over $k$ and can be embedded in a division ring
  containing a noncommutative free field.
\end{corollary}

\begin{proof}
  We adhere to the notation above and let $Q_q(2)$ denote
  the field of fractions of $M_q(2)$.
  Since $q$ is not a root of 1, the extension of $\sigma_b$ to $k(a)$
  is an automorphism of infinite order. By Theorem~\ref{le:theoremlorenz},
  $Q(R_b)=k(a)(b; \sigma_b)$ contains a noncommutative free
  $k$-subalgebra. Since $Q(R_b)$ is a subfield of $Q_q(2)$, it follows
  that $Q_q(2)$ contains a noncommutative free $k$-subalgebra and
  has infinite dimension over its centre $Z$,
  which is infinite because $q$ is not a root of 1. The completion of
  $Q_q(2)$ with respect to the valuation sending $d$ to $-1$ is
  the field of skew Laurent series $Q(R_c)((d^{-1};
  \sigma_d, \delta_d))$, which, by Theorem~\ref{th:freefield},
  contains a free field $\ff{Z}{}{X}$ on a countable set $X$.
\end{proof}

An application of Theorem \ref{th:ffpol} with a nonidentity
automorphism is the following. Let $k$ be a commutative field and
let $\lambda$ be a nonzero element in $k$. Let $B_{\lambda}$ be the
$k$-algebra generated by two elements $x$, $y$, together with their
inverses $x^{-1}$, $y^{-1}$, subjected to the relation $xy=\lambda
yx$. This algebra has been considered by Lorenz in \cite{mL84}. It
can be viewed as a factor $B_{\lambda}=kG/{\langle
(a,b)-\lambda\rangle}$ of the group algebra $kG$, where $G=\langle
a, \ b: (a,(a,b))=(b,(a,b))=1\rangle$ is the free nilpotent group of
class 2 on 2 generators, or as the localization of an Ore extension,
$B_{\lambda}=k[x, x^{-1}][y, y^{-1}; \sigma]$, where $\sigma$ is the
$k$-automorphism of $k[x]$ given by $\sigma(x)=\lambda x$. It
follows from this last characterization that $k[x][y;
\sigma]\subseteq B_{\lambda}\subseteq k(x)(y; \sigma)$ and hence
$B_{\lambda}$ is an Ore domain with field of fractions
$Q_{\lambda}=k(x)(y;\sigma)$. If $\lambda$ is not a root of unity
then $\sigma$ has infinite order and, by
Lemma~\ref{le:Laurentcentre}, $Z(Q_{\lambda})=\{r\in k(x) :
\sigma(r)=r\}=k$, which is an infinite field. Moreover, in this case,
$Q_{\lambda}$ is known to contain a noncommutative free $k$-algebra. It
follows that it is infinite dimensional over its centre. The order function $o$
on $k[x][y; \sigma]$ such that $o(y)=1$ is a valuation. It extends
to a valuation on $B_{\lambda}$ and, consequently, to one on
$Q_{\lambda}$. So, Theorem~\ref{th:ffpol} implies that with respect
to this valuation the completion $\widehat{Q_{\lambda}}=k(x)((
y;\sigma))$ contains a free field $\ff{k}{}{X}$, where $X$ is a
countable set. This example should be compared with the case
considered in the Section~\ref{sec:mn}, where it was proved
that the classic field of fractions of $kG$, for that same group
$G$, has a valuation with a completion containing a free field.

\section{A class of rings with $t$-adic valuation}\label{sec:adicval}

Let $R$ be a ring with a regular central element $t$ such that
$\bigcap{t^nR}=0$ and $R/tR$ is right Ore. In \cite{aL95},
Lichtman proved that $R$ can be embedded in a division ring $D$.
We will show that in some cases such a division ring $D$ contains
a free field. As a consequence we show that a free field can
be embedded in the completion of a field of fractions of the universal
enveloping algebra of a Lie algebra in characteristic zero.

We start by stating the main result of \cite{aL95}.

\begin{theorem}[{\cite[Theorem~1]{aL95}}]\label{th:valLicht}
  Let $R$ be a ring and let $t$ be a regular central element of $R$. Assume
  that $\bigcap{t^nR}=0$ and that $R/tR$ is a right Ore domain. Then
  $R$ can be embedded in a division ring $D$, the $t$-adic valuation
  $\nu_t$ of $R$ is extended to a valuation $\nu$ of $D$ and the
  subset $R(R^*)^{-1}$ is dense in $D$. Moreover, if $S$ is the
  valuation ring of $\nu$, then $S/tS$ is isomorphic to $\Delta$,
  the field of fractions of $R/tR$.\hfill\qed
\end{theorem}

In its proof it is shown that $t$ is a central
element of $D$, that $S$ is a local ring with maximal ideal $tS$ and
that $tS$ defines a $t$-adic valuation in $S$ which extends
$\nu_t$ and coincides with the restriction of $\nu$ to $S$.
Moreover, $D$ is defined as the localization $ST^{-1}$, where
$T=\{1,t,t^2,\dots\}$, and it is shown to have the property
that all of its elements have a unique representation of the
form
$$
 \sum_{i=r}^{\infty}d_it^i,
$$
where the $r$ is an integer and the $d_i$ belong to a system of
coset representatives of $S/tS$ in $S$. It follows that $D$ is complete with
respect to the valuation $\nu$.

\begin{theorem}
  Under the hypothesis of Theorem~\ref{th:valLicht}, if $\Delta$ is infinite-dimensional over
  $Z(\Delta)$, then $D$ contains a free field $\ff{Z(D)}{}{X}$, where $X$ is
  a countable set.
\end{theorem}

\begin{proof}
  We shall prove that under the given hypothesis $D$ is
  infinite-dimensional over $Z(D)$. Suppose that $[D:Z(D)]=n<\infty$. It will
  be shown that this implies that $[\Delta:Z(\Delta)]$ is also finite
  and $[\Delta:Z(\Delta)]\leq n$. Let $\delta_1,\dots,\delta_{n+1}$
  be a sequence of $n+1$ elements in $\Delta=S/tS$ and choose $s_1,
  \dots, s_{n+1}\in S$ such that $\delta_i=\overline{s_i}=s_i+tS$, for all $i=1,\dots, n+1$.
  Since $[D:Z(D)]=n$, the elements $s_1,\dots, s_{n+1}$ are linearly
  dependent over $Z(D)$. We can suppose, then, that
  there exist nonzero elements $z_1,\dots, z_m\in Z(D)$ such that
  $z_1s_1+\dots +z_ms_m=0$, for some $m\leq n+1$. For each $j=1,\dots, m$,
  let $r_j=\nu(z_j)$. Then $\nu(z_jt^{-r_j}) = \nu(z_j)-r_j\nu(t)=\nu(z_j)-r_j=0$.
  So, letting $u_j = z_jt^{-r_j}$, we get $z_j=u_jt^{r_j}$, for
  $j=1,\dots,m$, with $u_j\in S\setminus tS$. Our dependence
  relation now reads
  \begin{equation}\label{eq:adic}
  u_1s_1t^{r_1}+\dots+u_ms_mt^{r_m}=0.
  \end{equation}
  Let $r=\min\{r_j : j=1, \dots, m\}$ and suppose, without lost of generality,
  that $r=r_1$. Then $r_j-r\geq 0$ and hence $t^{r_j-r}\in R\subseteq S$,
  for all $1\leq j\leq m$. Multiplying \eqref{eq:adic} by $t^{-r}$ yields
  the following relation in $S$, $u_1s_1+u_2s_2t^{r_2-r}+\dots+u_ms_mt^{r_m-r}=0$.
  Thus, in the quotient $\Delta = S/tS$, we have
  $$
  \overline{u_1}\delta_1+\overline{u_2t^{r_2-r}}\delta_2+\dots+
  \overline{u_mt^{r_m-r}}\delta_m=0.
  $$
  Now since the $z_j$ and $t$ are central, $u_j\in S\cap Z(D)\subseteq Z(S)$.
  As we have seen above $u_1\in S\setminus tS$. So we have a nontrivial dependence
  relation among the $\delta_j$ with coefficients in $Z(\Delta)$. It follows
  that $[\Delta:Z(\Delta)]\leq n$. We conclude that $D$ is
  infinite-dimensional over $Z(D)$. Moreover, since $t$ is not invertible in $R$,
  the set $T$ of powers of $t$ is infinite and, therefore, the centre $Z(D)$ of $D$
  is infinite. Finally, since $\nu$ is a discrete valuation and
  $D$ is complete with respect to it, $D$ contains a free field $\ff{Z(D)}{}{X}$,
  where $X$ is a countable set, by Theorem \ref{th:freefield}.
\end{proof}

An important special case of Theorem~\ref{th:valLicht} is
provided by universal enveloping algebras of Lie algebras.
We follow \cite[Section~2.2]{aL95}: let $L$ be a Lie algebra over a commutative field
$k$ and let $U(L)$ be its universal enveloping algebra. It is known that $U(L)$ has a
filtration
$$
0\subseteq U_0\subseteq U_1\subseteq\dots\subseteq U_i\subseteq\dots,
$$
where $U_0=k$, $U_i=k+L+L^2+\dots +L^i$ and $L^i$ is the subspace
generated by all products of $i$ elements of $L$. By \cite[Proposition~2.6.1]{pC95},
this filtration defines a valuation $\nu$ on $U(L)$, called the \emph{canonical valuation}.
\cite[Theorem~2]{aL95} then asserts that $U(L)$ can be embedded in a skew field
$D$, which is complete in the topology defined by the canonical valuation. (This result
had been originally obtained by Cohn in \cite{pC61} with a different proof.) Let $D(L)$ be
the subfield of $D$ generated by $U(L)$. Since $D$ is complete, the closure
$\overline{D(L)}$ of the subspace $D(L)$ of the metric space
$D$ is complete, and hence, $\overline{D(L)}$ is the
completion of $D(L)$.

Lichtman also proved in \cite{aL99} that, if $\charac k=0$, then any
skew field that contains $U(L)$ contains a free $k$-algebra of
rank 2 and, therefore, must have infinite dimension over its centre.
So we have the following consequence of Theorem~\ref{th:freefield}.

\begin{theorem}
Let $L$ be a nonabelian Lie algebra over a commutative field of
characteristic zero and let $U(L)$ be its universal enveloping algebra.
In the notation above, both $D$ and $\overline{D(L)}$ contain a free field
$\ff{Z(D(L))}{}{X}$ on a countable set $X$.\hfill\qed
\end{theorem}

\section*{acknowledgments}
The authors are indebted to A.~Lichtman who communicated to them
the contents of Section~\ref{sec:noff} and kindly permitted them to
be included in the paper.



\begin{thebibliography}{99}

\bibitem{kC03}
  \textsc{K. Chiba},
  Free fields in complete skew fields and their valuations,
  \textit{J. Algebra}~\textbf{263} (2003), 75--87.

\bibitem{pC61}
  \textsc{P. M. Cohn},
  On the embedding of rings in skew fields,
  \textit{Proc. London. Math. Soc.}~\textbf{11} (1961), 511--530.

\bibitem{pC77}
  \bysame,
  \textit{Skew Field Constructions},
  Cambridge University Press, Cambridge, 1977.

\bibitem{pC85}
  \bysame,
  \textit{Free Rings and Their Relations},
  2nd. Ed., Academic Press, London, 1985.

\bibitem{pC95}
  \bysame,
  \textit{Skew Fields. Theory of General Division Rings},
  Cambridge University Press, Cambridge, 1995.


\bibitem{sJ55}
  \textsc{S. A. Jennings},
  The group ring of a class of infinite nilpotent groups,
  \textit{Can. J. Math.}\textbf{7} (1955),~169--187.

\bibitem{cK95}
  \textsc{C. Kassel},
  \textit{Quantum Groups}, Springer-Verlag, New York, 1995.

\bibitem{tL91}
  \textsc{T. Y. Lam},
  \textit{A First Course in Noncommutative Rings},
  Springer-Verlag, New York, 1991.

\bibitem{aL95}
  \textsc{A. I. Lichtman},
  Valuation methods in division rings,
  \textit{J. Algebra}~\textbf{177} (1995), 870--898.

\bibitem{aL99}
  \bysame,
  Free subalgebras in division rings generated by universal enveloping algebras,
  \textit{Algebra Colloq.}~\textbf{6} (1999), 145--153.

\bibitem{mL84}
  \textsc{M. Lorenz},
  Group rings and division rings,
  \textit{Methods in Ring Theory (Antwerp, 1983)}, 265--280,
  NATO Adv. Sci. Inst. Ser. C Math. Phys. Sci., 129, Reidel, Dordrecht, 1984.

\bibitem{mL86}
  \bysame,
  On free subalgebras of certain division algebras,
  \textit{Proc. Amer. Math. Soc.}~\textbf{98} (1986), 401--405.

\bibitem{lM83}
  \textsc{L. Makar-Limanov},
  The skew field of fractions of the Weyl algebra contains a
  free noncommutative subalgebra,
  \textit{Comm. Alg.}~\textbf{11} (1983), 2003--2006.

\bibitem{lM84}
  \bysame, On group rings of nilpotent groups,
  \textit{Israel J. Math.}~\textbf{48} (1984), 244-248.

\bibitem{dP89}
  \textsc{D. S. Passmann},
  \textit{Infinite Crossed Products},
  Academic Press, Inc., Boston, MA, 1989.


\end{thebibliography}
\end{document}